\documentclass{amsart}
\usepackage{amsmath, amssymb, latexsym}
\title{Extremal Betti numbers of Rips complexes}
\author{Michael Goff}

\newtheorem{theorem}{Theorem}[section]

\newtheorem{corollary}[theorem]{Corollary}
\newtheorem{lemma}[theorem]{Lemma}

\newcommand{\K}{\Gamma}
\newcommand{\field}{{\bf k}}

\newcommand{\lk}{\mbox{\upshape lk}\,}

\newcommand{\st}{\mbox{\upshape st}\,}
\newcommand{\ds}{\mbox{\upshape dist}\,}

\newcommand{\B}{\tilde{\beta}}
\newcommand{\HH}{\tilde{H}}

\def\proofof#1{\smallskip\noindent {\it Proof of #1: \ }}
\def\endproof{\hfill$\square$\medskip}

\begin{document}

\begin{abstract}
Upper bounds on the topological Betti numbers of Vietoris-Rips complexes are established, and examples of such complexes with high Betti numbers are given.
\end{abstract}

\date{August 27, 2009}

\maketitle

\section{Introduction}
In this paper we consider extremal Betti numbers of Vietoris-Rips complexes.  Given a finite set of points $S$ in Euclidean space $\mathbb{R}^d$, we define the \textit{Vietoris-Rips complex} $R^{\epsilon}(S)$, or \textit{Rips complex}, as the simplicial complex whose faces are given by all subsets of $S$ with diameter at most $\epsilon$.  Take $R(S) := R^1(S)$.  Our main goal in this paper is to determine the largest topological Betti numbers of $R(S)$ in terms of $|S|$ and $d$.

Rips complexes have a wide range of applications.  Vietoris \cite{Vietoris} used Rips complexes to calculate the homology groups of metric spaces.  Other applications include geometric group theory \cite{GeoGroup}, simplicial approximation of point-cloud data \cite{PointModel1}, \cite{PointModel2}, \cite{PointModel3}, \cite{PointModel4}, and modeling communication between nodes in sensor networks \cite{Sensor1}, \cite{Sensor2}, \cite{Sensor3}.  In the specific case of the Euclidean plane, the topology of Rips complexes is studied in \cite{Planar}.  Rips complexes are used in manifold reconstruction in \cite{ChazOud}.

One of the main uses of the Rips complex is to approximate the topology of a point cloud.  The point cloud might be a random sample of points from a manifold or some other topological space.  Several papers, such as \cite{ChazOud}, give conditions on the point sample under which the Rips complex can be used to determine the homology and homotopy groups of the underlying space.  It is generally assumed that the Rips complex $R^{\epsilon}(S)$ is chosen in such a way that the points of $S$ are dense in the underlying space, relative to $\epsilon$.

For a fixed base field $\field$, we denote the homology groups of a simplicial complex $\Gamma$ by $\HH_p(\Gamma;\field)$.  The topological Betti numbers are given by $\B_p(\Gamma;\field) := \dim_{\field}(\HH_p(\Gamma;\field))$.  All of our results are independent of $\field$, and so from now on we suppress the base field from our notation.  We define $$M_{p,d}(n) := \max\{\B_p(R(S)): S \subset \mathbb{R}^d, |S| \leq n\}.$$

The \v{C}ech complex is another simplicial complex that captures the topology of a point cloud.  Given $S \subset \mathbb{R}^d$, the \v{C}ech complex $C(S)$ has vertex set $S$ and faces given by all sets of points that are contained in a ball of radius $\epsilon/2$.  By the Nerve Lemma \cite{Bjorner}, $\B_k(C(S)) = 0$ for $k \geq d$.  By contrast, if $d \geq 2$, then $\B_k(R(S))$ can be nonzero for arbitrarily large $k$.

[To be added: discussion of Matt Kahle's work]

In the interest of understanding the topology of Rips complexes, we consider the largest possible topological Betti numbers.  We find that nontrivial upper bounds are possible, but also that the Betti numbers can be quite large under specialized constructions.

The structure of this paper is as follows.  We review some facts on simplicial complexes in Section \ref{prelim}.  In Section \ref{Hom1}, we prove that $M_{1,d}(n)$ grows linearly in $n$ for each fixed $d$.  In Section \ref{Hom2}, we prove that $M_{2,2}(n)$ grows linearly in $n$, and in general, for each fixed $\delta$ and $d$, $M_{2,d}(n) < \delta n^2$ for sufficiently large $n$.  We also give a construction to prove that $M_{2,5}(n) > Cn^{3/2}$ for some constant $C$ and sufficiently large $n$.  In Section \ref{Hom3}, we extend the results of the previous sections by showing that for each fixed $\delta, p, d$, $M_{p,d}(n) < \delta n^p$ for sufficiently large $n$, and also that $M_{p,5}(n) > C_pn^{p/2+1/2}$, for a value $C_p$ that depends only on $p$ and sufficiently large $n$.  In Section \ref{QuasiRips} we consider similar bounds on the Betti numbers of related objects known as quasi-Rips complexes.  Our proofs make frequent use of the Mayer-Vietoris sequence and a careful analysis of the structure of the first homology group of a Rips complex.

\section{Definitions and preliminaries}
\label{prelim}
An abstract simplicial complex $\Gamma$ on a finite set $S$, called the vertex set, is a collection of subsets, called \textit{faces}, of $S$ that is closed under inclusion and contains all singleton subsets.  A face with two elements is called an \textit{edge}.  For convenience, we generally suppress commas and braces when expressing faces of a simplicial complex.  We also refer to the vertex set of $\Gamma$ by $V(\Gamma)$.

If $F$ is a face of $\Gamma$, then we define the link $\lk_\Gamma(F)$, or $\lk(F)$ when $\Gamma$ is implicit, as $\{G \in \Gamma: G \cup F \in \Gamma, G \cap F = \emptyset\}$.  The star $\st_\Gamma(F) = \st(F)$ is $\{G \in \Gamma: G \cup F \in \Gamma\}$.  If $\Gamma$ is a Rips complex $R(S)$, then the stars and links are also Rips complexes.  For an arbitrary subset $F \subset S$, define $N(F) := \{v \in S-F: \ds(u,v) \leq 1$ for all $u \in F\}$.  Then for $F \in R(S)$, $\lk(F) = R(N(F))$ and $\st(F) = R(N(F) \cup F)$.  The \textit{induced subcomplex} $\Gamma[W]$ for $W \subset V(\Gamma)$ is defined as $\{F: F \in \Gamma, F \subset W\}$.  For a Rips complex $R(S)$, $R(S)[W] = R(W)$.

Every Rips complex is also a flag complex.  A \textit{flag} complex, also called a \textit{clique} complex, is a simplicial complex $\Gamma$ such that $F \in \Gamma$ whenever all $2$-subsets of $F$ are edges in $\Gamma$.  Thus a flag complex is determined by its edges.  For a graph $G$, we define $X(G)$ to be the unique flag simplicial complex with the same edges as $G$.

Let $\Gamma$ be a simplicial complex with a subcomplex $\Gamma'$.  Let $\phi: \HH_p(\Gamma') \rightarrow \HH_p(\Gamma)$ be the map on homology induced by inclusion.  We define $\Omega_{p}(\Gamma,\Gamma')$ to be the image of $\phi$.

Our proofs give special attention to the structure of the first homology group.  Given a simplicial complex $\Gamma$ with $\{v_1, \ldots, v_r\} \subset V(\Gamma)$ and edges $v_1v_2, \ldots, v_{r-1}v_r, v_rv_1$, the notation $C = (v_1, \ldots, v_r)$ refers to the graph theoretic cycle in $\Gamma$.  Taking subscripts mod $r$, we equivalently think of $C$ as the simplicial $1$-chain $\sum_{i=1}^r \pm v_iv_{i+1}$, with signs chosen so that $\partial C = 0$.  We denote by $[C]_{\Gamma}$, or $[C]$ when $\Gamma$ is clear from context, the equivalence class of $C$ in $\HH_1(\Gamma)$.
\begin{lemma}
\label{CycleBasis}
There is a basis for $\HH_1(\Gamma)$ such that every element of the basis is the equivalence class of a simple, chord-free cycle.
\end{lemma}
\begin{proof}
It is a standard fact in algebraic topology that $\HH_1(\Gamma)$ has a basis of equivalence classes of cycles.  Let $B$ be such a basis.  If $[C] \in B$ is the equivalence class of a non-simple cycle of the form $C = (v_1, \ldots, v_r, v_1, v'_2, \ldots, v'_{r'})$, then replace $[C]$ in $B$ by $[C_1] = [(v_1, \ldots, v_r)]$ and $[C_2] = [v_1, v'_2, \ldots, v'_{r'}]$.  Also, if $C = (v_1, \ldots, v_r)$ and $C$ has a chord $v_iv_j$, then replace $[C]$ by $[C_1] = [(v_1, \ldots, v_i,v_j, \ldots, v_r)]$ and $[C_2] = [(v_i, \ldots, v_j)]$.  Then delete elements from $B$ until $B$ is again a basis for $\HH_1(\Gamma)$.  Repeat this operation until all elements of $B$ are equivalence classes of simple, chord-free cycles.
\end{proof}

\section{Results on $M_{1,d}(n)$ and lemmas}
\label{Hom1}

In this section we prove a linear upper bound on $M_{1,d}(n)$, and we also give some lemmas on the structure of $\HH_1(R(S))$.  Those lemmas are needed to prove results on higher homology.  Before the main theorem of this section, we need a general fact on the homology of simplicial complexes.

\begin{lemma}
\label{LinkLemma}
Consider $v \in S$.  Then for all $p$, $\B_p(R(S)) \leq \B_p(R(S-v)) + \B_{p-1}(\lk(v))$.
\end{lemma}
\begin{proof}
Consider $\Delta := R(S-v)$ and $\Delta' = \st_{R(S)}(v)$.  Then $\Delta \cup \Delta' = R(S)$ and $\Delta \cap \Delta' = \lk(v)$.  Since $\Delta'$ is a cone, that is, $v$ is contained in all maximal faces of $\Delta'$, all of its homology groups vanish.  The lemma then follows from the Mayer-Vietoris sequence.
\end{proof}

\begin{theorem}
For every $d$, there exists a constant $C_d$ such that $M_{1,d}(n) \leq C_dn$.
\end{theorem}
\begin{proof}
Let $B^d$ be a closed ball of radius $1$ in $\mathbb{R}^d$, and let $$C_d := \max\{|T|: T \subset B^d, \ds(u,v) > 1 \mbox{ for all } u,v \in T\}-1.$$ Choose $v \in S$.  Then $\lk(v) = R(N(v))$ is a Rips complex on a point set contained in a ball of radius $1$.  Suppose that $R(N(v))$ has $k$ connected components with representative vertices $v_1, \ldots, v_k$.  Then for all $1 \leq i < j \leq k$, $\ds(v_i,v_j) > 1$.  Thus $k \leq C_d+1$, and $\B_0(R(N(v))) \leq C_d$.

We prove the theorem by induction on $n$, with the base case $n=0$ evident.  By the inductive hypothesis, $\B_1(R(S-v)) \leq C_d(n-1)$.  Also, $B_0(R(N(v))) \leq C_d$.  The result follows from Lemma \ref{LinkLemma}.
\end{proof}

Our next lemma relates the first Betti number of the clique complex of a certain kind of graph to the zeroth Betti number of a related graph.  In the following, we may think of $X(G)$ as $R(U \sqcup V)$, where $U$ and $V$ are both clusters of points of diameter at most $1$.

\begin{lemma}
\label{BipartiteLemma}
Let $G$ be a graph with vertex set $U \sqcup V$ such that all edges $uu'$ and $vv'$ are in $G$ for $u,u' \in U$, $v,v' \in V$.  Let $G'$ be the bipartite graph on $U \sqcup V$ obtained from $G$ by deleting all $uu'$ and $vv'$ for $u,u' \in U$ and $v,v' \in V$, and then deleting any isolated vertices.  Then $\B_1(X(G)) = \B_0(G')$.  Let $u_iv_i$, $u_i \in U, v_i \in V$, $1 \leq i \leq q$ be a set of representative edges of the $q$ components of $G'$.  Then the cycles $[(u_1,u_i,v_i,v_1)]$ for $2 \leq i \leq q$ can be taken as a basis for $\HH_1(X(G))$.
\end{lemma}
\begin{proof}
Suppose that $G'$ has $q$ connected components.  We show that $\B_1(X(G)) = \B_0(G') = \max\{0,q-1\}$ by induction on $q$.  In the case that $q=0$, $X(G)$ is the disjoint union of simplices on $U$ and $V$, and so $\B_1(X(G)) = 0$.

Next we show that $\B_1(X(G)) = 0$ if $q=1$.  Enumerate the edges of $G'$ by $e_1, \ldots, e_z$ in such a way that for all $i > 1$, $e_i$ shares an endpoint with some previous edge.  For all $i$, construct $G_i$ from $G$ by removing $e_{i+1},\ldots,e_z$ from $G$.  Note that $G_z = G$.  Since $X(G_1)$ consists of two disjoint simplices connected by a single edge, $\B_1(X(G_1))=0$.  We show by induction on $i$ that $\B_1(X(G_i))=0$ for all $i$, and in particular that $\B_1(X(G)) = 0$.

Let $C$ be a graph theoretic cycle in $G_i$ and consider $[C]_{X(G_i)}$ for $i > 1$.  Let $e_i = uv$ for $u \in U, v \in V$, and suppose without loss of generality (perhaps by switching the roles of $U$ and $V$) that $G_i$ contains an edge $uv'$ for some $v \neq v' \in V$.  This assumption is valid by the assumption that $e_i$ shares an endpoint with $e_j$ for some $j < i$.  If $C$ contains $uv$, let $C'$ be the cycle obtained by replacing $uv$ in $C$ by the two edges $uv', v'v$.  Otherwise, set $C' := C$.  Since $C'$ avoids $e_i$, $C'$ is a cycle in $G_{i-1}$.  Then $[C']_{X(G_{i-1})} = 0$ by the inductive hypothesis, and hence $[C']_{X(G_i)} = 0$ by $X(G_{i-1}) \subset X(G_i)$.  We have $uvv' \in X(G_i)$ by the flag property, and so $[C]_{X(G_i)}=[C']_{X(G_i)} = 0$.  This proves that $\B_1(X(G_i)) = 0$.

Now suppose that $q \geq 2$, and let $W$ be the vertex set of a component of $G'$.  Let $\tilde{G}$ be obtained from $G$ by removing the edges of $G$ with one endpoint in $U \cap W$ and the other in $V \cap W$.  Set $\Delta := X(\tilde{G})$.  Then $\Delta$ is connected and satisfies $\B_1(\Delta) = q-2$ by the inductive hypothesis.  Set $\Delta' := X(G)[W]$.  By the $q=1$ case, $\B_1(\Delta') = 0$, and also $\Delta'$ is connected.  We also have that $\Delta \cup \Delta' = X(G)$ and $\Delta \cap \Delta'$ is the disjoint union of simplices on $W \cap U$ and $W \cap V$.  Hence $\B_0(\Delta \cap \Delta') = 1$ and all other Betti numbers of $\Delta \cap \Delta'$ vanish.  Apply the portion of the Mayer-Vietoris sequence with components $\Delta$ and $\Delta'$ $$0 \rightarrow \HH_1(\Delta) \stackrel{\phi}{\rightarrow} \HH_1(X(G)) \stackrel{\partial}{\rightarrow} \HH_0(\Delta \cap \Delta') \rightarrow 0$$ to conclude that $\B_1(X(G)) = q-1$.

Now we prove that the cycles $[(u_1,u_i,v_i,v_1)]$ for $2 \leq i \leq q$ can be taken as a basis for $\HH_1(X(G))$ by induction on $q$, with the cases $q=0$ and $q=1$ trivial.  Assume that $u_q,v_q \in W$, with $W$ as above.  Note that the homology groups in the above Mayer-Vietoris sequence are vector spaces, and hence the sequence splits.  Since the inclusion-induced map $\phi$ is injective, the set of cycles $\{[(u_1,u_i,v_i,v_1)]\}$ for $2 \leq i \leq q-1$ is a basis for $\Omega_1(X(G),\Delta)$.  Also, by the structure of the connecting homomorphism, $\partial([(u_1,u_q,v_q,v_1)]) = \pm [v_q - u_q]$ is a nonzero element of $\HH_0(\Delta \cap \Delta')$.  This proves the result.
\end{proof}

\begin{corollary}
\label{BipartiteGen2}
Let all quantities be as in Lemma \ref{BipartiteLemma}, and suppose that $X(G)$ is an induced subcomplex of some larger complex $\K$.  Then there exists an edge set $\{u_iv_i: u_i \in U, v_i \in V, 1 \leq i \leq q'\}$ for some $q' \leq q$, such that each edge is in a different component of $G'$ and the set of cycles $\{[(u_1,u_i,v_i,v_1)]\}$ for $2 \leq i \leq q'$ is a basis for $\Omega_1(\K,X(G))$.
\end{corollary}
\begin{proof}
Take the set of cycles from Lemma \ref{BipartiteLemma} and reduce it to a linearly independent set in $\Omega_1(\K,X(G))$ with the same span.
\end{proof}

Now we begin constructing our regular form of a basis for $\HH_1(R(S))$.  For a given $\epsilon > 0$, we partition $\mathbb{R}^d$ into $\epsilon$-cubes.  We say that $K \subset \mathbb{R}^d$ is an $\epsilon$-\textit{cube} if there exist integers $m_1, \ldots, m_d$ such that $K$ is the product of half-open intervals $[m_1 \epsilon, (m_1+1)\epsilon) \times \ldots \times [m_d \epsilon, (m_d+1)\epsilon)$.  If $\epsilon \leq d^{-1/2}$ and $S$ is a finite subset of some $\epsilon$-cube $K$, then $R(S)$ is a simplex.

The next lemma gives our first form for a basis of $\HH_1(R(S))$.  Call a basis of the prescribed form $C_{d,r,\epsilon}$-\textit{regular}.

\begin{lemma}
\label{RegularGenerators}
Let $S$ be a finite subset of $\mathbb{R}^d$ contained in a ball $D$ of radius $r$, and fix $\epsilon \leq d^{-1/2}$.  Then there exists a constant $C_{d,r,\epsilon}$, which depends only on $d$, $r$, and $\epsilon$, such that the following holds.  There exists a basis of $\HH_1(R(S))$ such that all but at most $C_{d,r,\epsilon}$ of the basis elements are of the form $[(u,u',v',v)]$, where $u$ and $u'$ are in the same $\epsilon$-cube, and $v$ and $v'$ are in the same $\epsilon$-cube.
\end{lemma}

If a cycle $C = (u,u',v',v)$ satisfies the condition that $u$ and $u'$ are in the same $\epsilon$-cube, and $v$ and $v'$ are in the same $\epsilon$-cube, then we say that $C$ is $\epsilon$-\textit{simple}.

\begin{proof} There is a set $\mathcal{K} = \{K_1, \ldots, K_\kappa\}$ of $\kappa := (\lceil 2r/\epsilon \rceil +1)^d$ $\epsilon$-cubes that cover $S$.  Choose a basis $B$ of $\HH_1(R(S))$ so that each basis element is the equivalence class of a simple, chord-free graph theoretic cycle in $R(S)$, as allowed by Lemma \ref{CycleBasis}.  Given three points $u,v,w \in S \cap K_i$ for some $i$, $R(u,v,w)$ is a simplex.  Hence, given a cycle $[C] \in B$, $C$ contains at most $2$ vertices in $K_i$, which implies that $C$ contains at most $2 \kappa$ vertices in total.  We say that two cycles $C$ and $C'$ are \textit{near} each other if, by labeling vertices appropriately, $C=(v_1,\ldots,v_k)$, $C'=(v_1',\ldots,v_k')$, and for all $i$, $v_i$ and $v_i'$ are in the same $\epsilon$-cube.  Nearness is an equivalence relation.  There are at most $C_{d,r,\epsilon} := \sum_{i=1}^{2\kappa}\kappa^i$ nearness equivalence classes for simple, chord-free cycles of length at most $2 \kappa$ in $D$.

Suppose that $[C]=[(v_1,\ldots,v_k)] \in B$ and $[C']=[(v_1',\ldots,v_k')] \in B$ are near each other and are not $\epsilon$-simple.  The following subscripts are understood mod $k$.  Then $[C'] = [C] + \sum_{i=1}^k [(v_i',v_{i+1}',v_{i+1},v_i)]$.  Remove $[C']$ from $B$ and add each of the $[(v_i',v_{i+1}',v_{i+1},v_i)]$ to $B$.  Then reduce $B$ to a basis for $\HH_1(R(S))$ by removing elements that are linear combinations of other elements.  After this reduction, all elements of $B$ are equivalence classes of simple, chord-free cycles; the reason is that if $(v_i',v_{i+1}',v_{i+1},v_i)$ has a chord, then $[(v_i',v_{i+1}',v_{i+1},v_i)] = 0$ by the fact that $R(S)$ is flag.  This operation strictly decreases the number of non-$\epsilon$-simple elements of $B$ while maintaining the span of $B$.  Repeat this operation as many times as possible; then $B$ contains at most $C_{d,r,\epsilon}$ non-$\epsilon$-simple generators.
\end{proof}

We further refine our basis for $\HH_1(R(S))$.  Let $K$ be a distinguished $\epsilon$-cube and $W := K \cap S$.  We say that a basis $B$ for $\HH_1(R(S))$ is $W$-\textit{regular} if all but $C_{d,r,\epsilon}+{\kappa \choose 2}$ elements $[C] \in B$ are of one of the following two forms. \newline
1) $C = (w,w',v',v)$ with $w,w' \in W$ and $v,v'$ in the same $\epsilon$-cube. \newline
2) $C = (u,u',v',v)$ with $u,u'$ in the same $\epsilon$-cube and $v,v'$ in the same $\epsilon$-cube, and furthermore there is no face $wuv$ or $wu'v'$ for any $w \in W$.

\begin{lemma}
\label{RegGen2}
Let $S$ be as in Lemma \ref{RegularGenerators}, and let $W$ be the intersection of a fixed $\epsilon$-cube with $S$.  Then $\HH_1(R(S))$ has a $W$-regular basis.
\end{lemma}
\begin{proof}
Let $\mathcal{K} = \{K_1, \ldots, K_\kappa\}$ be a set of $\epsilon$-cubes that cover the points of $S$, with $K=K_1$ if $W \neq \emptyset$.  By Lemma \ref{RegularGenerators}, the equivalence classes of $\epsilon$-simple cycles in $R(S)$ span a subspace $\Omega$ of $\HH_1(R(S))$ with $\dim(\Omega) \geq \B_1(R(S)) - C_{r,d,\epsilon}$.  It is clear that $$\Omega \subseteq \sum_{i,j}\Omega_1(R(S),R(S \cap (K_i \cup K_j))),$$ and since each $\Omega_1(R(S),R(S \cap (K_i \cup K_j)))$ is spanned by $\epsilon$-simple cycles by Corollary \ref{BipartiteGen2}, $$\Omega \supseteq \sum_{i,j}\Omega_1(R(S),R(S \cap (K_i \cup K_j))).$$ We first construct a basis $B$ for $\Omega$ as follows.  By Corollary \ref{BipartiteGen2}, for all $1 \leq i < j \leq \kappa$ we may choose integers $p_{i,j}$ and a basis $B_{i,j}$ for $\Omega_1(R(S),R(S \cap (K_i \cup K_j)))$ given by $\{[(u^{i,j}_1,v^{i,j}_1,v^{i,j}_k,u^{i,j}_k)], 2 \leq k \leq p_{i,j}\}$ with the properties prescribed in Corollary \ref{BipartiteGen2}.  Then let $B$ be a linearly independent subset of $\cup_{1 \leq i < j \leq \kappa}B_{i,j}$ with the same span.  If $W = \emptyset$, then an extension of $B$ to a basis for $\HH_1(R(S))$ is $W$-regular, as every element of $B$ is satisfies the second condition in the definition of a $W$-regular basis.  So now suppose that $W \neq \emptyset$.

Suppose that there exist distinct $w, w' \in W$ so that for some $1 \leq k < k' \leq p_{i,j}$, there exist faces $wu^{i,j}_kv^{i,j}_k$ and $w'u^{i,j}_{k'}v^{i,j}_{k'}$.  Then $$[(u^{i,j}_k,v^{i,j}_k,v^{i,j}_{k'},u^{i,j}_{k'})] = [(u^{i,j}_k,w,v^{i,j}_k,v^{i,j}_{k'},w',u^{i,j}_{k'})].$$ By the existence of the edge $ww'$, this is $[(u^{i,j}_{k'},u^{i,j}_{k},w,w')] + [(v^{i,j}_{k},v^{i,j}_{k'},w',w)]$.  Then replace $[(u^{i,j}_1,v^{i,j}_1,v^{i,j}_{k'},u^{i,j}_{k'})]$ with $[(u^{i,j}_{k'},u^{i,j}_{k},w,w')]$ and $[(v^{i,j}_{k},v^{i,j}_{k'},w',w)]$ in $B$, and then remove elements from $B$ until the set is linearly independent with the same span.  This operation does not decrease $|B|$, and it strictly decreases the number of elements $[C] \in B$ such that $C$ does not have a vertex in $W$.  Redefine variables so that $B_{i,j}$ is again of the form $\{(u^{i,j}_1,v^{i,j}_1,v^{i,j}_k,u^{i,j}_k)$, $2 \leq k \leq p_{i,j}\}$ for a new value of $p_{i,j}$.  Repeat this operation as many times as possible.

Also, there cannot exist $w \in W$ such that there are faces $wu^{i,j}_kv^{i,j}_k$ and $wu^{i,j}_{k'}v^{i,j}_{k'}$ for $k \neq k'$, since in that case faces $wu^{i,j}_ku^{i,j}_{k'}$ and $wv^{i,j}_kv^{i,j}_{k'}$ also exist and $[(u^{i,j}_k,v^{i,j}_k,v^{i,j}_{k'},u^{i,j}_{k'})]=0$.  If $k=1$, this violates the basis assumption.  If $k>1$,
\begin{eqnarray*}
& & [(u^{i,j}_1,v^{i,j}_1,v^{i,j}_{k},u^{i,j}_{k})] = \\
& & [(u^{i,j}_1,v^{i,j}_1,v^{i,j}_{k},u^{i,j}_{k})] + [(u^{i,j}_k,v^{i,j}_k,v^{i,j}_{k'},u^{i,j}_{k'})] = \\
& & [(u^{i,j}_1,v^{i,j}_1,v^{i,j}_{k},v^{i,j}_{k'},u^{i,j}_{k'},u^{i,j}_{k})].
\end{eqnarray*}
By the existence of faces $u_1^{i,j}u_k^{i,j}u_{k'}^{i,j}$ and $v_1^{i,j}v_k^{i,j}v_{k'}^{i,j}$, this implies that $[(u^{i,j}_1,v^{i,j}_1,v^{i,j}_{k},u^{i,j}_{k})] = [(u^{i,j}_1,v^{i,j}_1,v^{i,j}_{k'},u^{i,j}_{k'})]$, also a contradiction to the basis assumption.

We conclude that for each fixed pair $(i,j)$, there exists at most one value of $k$ such that there exists $w \in W$ and a face $wu^{i,j}_kv^{i,j}_k$.  If such a face exists and $k \neq 1$, then remove $[(u^{i,j}_1,v^{i,j}_1,v^{i,j}_{k},u^{i,j}_{k})]$ from $B$.  If $k=1$, note that $$[(u^{i,j}_2,v^{i,j}_2,v^{i,j}_{k'},u^{i,j}_{k'})] = -[(u^{i,j}_1,v^{i,j}_1,v^{i,j}_{2},u^{i,j}_{2})] + [(u^{i,j}_1,v^{i,j}_1,v^{i,j}_{k'},u^{i,j}_{k'})]$$ by the existence of faces $u^{i,j}_1u^{i,j}_2u^{i,j}_{k'}$ and $v^{i,j}_1v^{i,j}_2v^{i,j}_{k'}$.  Then replacing $[(u^{i,j}_1,v^{i,j}_1,v^{i,j}_{k'},u^{i,j}_{k'})]$ by $[(u^{i,j}_2,v^{i,j}_2,v^{i,j}_{k'},u^{i,j}_{k'})]$ for all $k' > 2$ and removing $[(u^{i,j}_1,v^{i,j}_1,v^{i,j}_{2},u^{i,j}_{2})]$ decreases $|B|$ by $1$ and preserves linear independence of $B$.  Doing this for all $1 \leq i < j \leq r$, $|B|$ decreases by at most ${\kappa \choose 2}$.  Then extend $B$ to a basis for $\HH_1(R(S))$.  This proves the result.
\end{proof}

We need yet another refinement of our basis.  We say that $B$ is a $W$-\textit{strongly regular} basis if the following holds.  For every pair of $\epsilon$-cubes $K_i$ and $K_j$ such that $R(S)$ has an edge with one endpoint in $K_i$ and another in $K_j$, choose a distinguished edge $u^{i,j}v^{i,j}$ with $u^{i,j} \in K_i, v^{i,j} \in K_j$.  Then all but $C_{d,r,\epsilon}+{\kappa \choose 2}$ elements of $B$ satisfy one of the two conditions in the definition of a $W$-regular basis and are also of the form $[(u^{i,j},v^{i,j},v',u')]$ for some $u' \in K_i$ and $v' \in K_j$.  Next we verify that $\HH_1(R(S))$ has a $W$-strongly regular basis.

\begin{lemma}
\label{RegGen3}
Let $S$ be as in Lemma \ref{RegularGenerators}, and let $W$ be the intersection of fixed $\epsilon$-cube with $S$.  Then $\HH_1(R(S))$ has a $W$-strongly regular basis.
\end{lemma}
\begin{proof}
First construct a $W$-regular basis $B'$, as guaranteed by Lemma \ref{RegGen2}, and we modify it into a strongly regular basis.  Let all quantities be as in the proof of Lemma \ref{RegGen2}.  For $2 \leq i < j \leq \kappa$, or for $1 \leq i < j \leq \kappa$ in the case that $W = \emptyset$, we may take $u^{i,j} := u^{i,j}_1$ and $v^{i,j} := v^{i,j}_1$, and all elements of $B$ with endpoints in $K_i$ and $K_j$ are of the form $(u^{i,j},v^{i,j},v',u')$ by construction.  This completes the proof in the case that $W = \emptyset$, and so now we assume that $W \neq \emptyset$ and $K = K_1$.

Now consider $1 = i < j \leq \kappa$.  Let $[C_1], \ldots, [C_t]$ be the elements of $B$ with vertices in $K_i$ and $K_j$, and define $C_k := (u_k,v_k,v_k',u_k')$ with $u_k,u_k' \in K_i$ and $v_k,v_k' \in K_j$ for $1 \leq k \leq t$.  For $2 \leq k \leq t$, add the cycles $[C_k'] := [(u_1,v_1,v_k,u_k)]$ and $[C_k''] := [u_1,v_1,v_k',u_k']$ to $B$, and remove $[C_k]$.  Observe that $C_k'$ and $C_k''$ satisfy Condition 1 in the definition of a $W$-regular basis.  Then remove any element from $B$ that can be written as a linear combination of other elements in $B'$, and repeat this operation as many times as possible.  By the existence of faces $u_1u_ku_k'$ and $v_1v_kv_k'$, $[C_k] = [C_k''] - [C_k']$, and therefore this operation preserves the property that $B'$ is a basis and hence $|B'|$ is preserved.  Since the operation also preserves $|B'-B|$, $|B|$ is preserved as well.  The lemma follows by taking $u^{i,j} := u_1$ and $v^{i,j} := v_1$.
\end{proof}

In order to obtain a more useful combinatorial picture of our $W$-strongly regular basis, we associate with the basis a set of edges with specific properties.  This set of edges will be instrumental in the proofs of later theorems.

\begin{corollary}
\label{EdgeSet}
Let all quantities be as in the statement and proof of Lemma \ref{RegGen3}.  There exists a set of edges $E=E(S) \subset \Gamma$, $|E| \geq \B_1(R(S)) - C_{r,d,\epsilon} - {\kappa \choose 2}$, which can be partitioned into sets $\{E_{i,j}\}$ for all pairs $1 \leq i < j \leq \kappa$, with the following properties. \newline
1) All the edges in $E_{i,j}$ are of the form $uv$ with $u \in S \cap K_i$ and $v \in S \cap K_j$. \newline
2) If $i \neq 1$, then there is no face $wuv$ for any $w \in S \cap K_1$ and $uv \in E_{i,j}$. \newline
3) Let $G^{i,j}$ be the bipartite graph that is the graph of $R(S \cap (K_i \cup K_j))$ with all edges in $R(S \cap K_i)$ and in $R(S \cap K_j)$ removed and then all isolated vertices removed.  Then $E_{i,j}$ does not contain two edges from the same component in $G^{i,j}$. \newline
4) Let $e_1, e_2$ be two edges in $E_{i,j}$, and let $G^{i,j}_1$ and $G^{i,j}_2$ be the components of $G^{i,j}$ that contain $e_1$ and $e_2$ respectively.  Then there is no vertex $w \in S \cap K_1$ such that $\lk(w)$ contains edges both in $G^{i,j}_1$ and $G^{i,j}_2$.
\end{corollary}
\begin{proof}
Let $B'$ be a $W$-strongly regular basis for $\HH_1(R(S))$, and let $B$ be as in the proof of Lemma \ref{RegGen3}.  For fixed $i<j$, let $\{[(u_1,v_1,v_k,u_k)] : k \leq 2\}$ be the set of elements of $B$ with $u_1,u_k \in K_i$ and $v_1,v_k \in K_j$.  Set $E_{i,j} := \{u_2v_2, \ldots, u_kv_k\}$.  By construction, $E$ satisfies Conditions 1 and 2.

To verify Condition 3, note that if $u,u' \in K_i$, $v,v',v'' \in K_j$, and $uv, u'v', u'v''$ are all edges in $R(S)$, then $[(u,v,v',u')] = [(u,v,v'',u')]$ by the existence of faces $vv'v''$ and $u'v'v''$ in $R(S)$.  By repeated applications of this fact, perhaps switching the roles of $K_i$ and $K_j$, we have that if $k' > k > 1$, then $[(u_1,v_1,v_k,u_k)] = [(u_1,v_1,v_{k'},u_{k'})]$ if $u_kv_k$ and $u_{k'}v_{k'}$ are in the same component in $G^{i,j}$.  This contradicts the linear independence of $B$, and so we have that all the edges in $E_{i,j}$ are in different components of $G^{i,j}$.

Now we verify Condition 4.  Let all quantities be as in the previous paragraph.  Suppose that $\lk(w)$ contains edges $u_k'v_k'$ and $u_{k'}'v_{k'}'$ in the same components of $G^{i,j}$ as $u_kv_k$ and $u_{k'}v_{k'}$ respectively.  By the argument of the previous paragraph and existences of faces $wu_k'v_k', wu_{k'}'v_{k'}', wu_k'u_{k'}', wv_k'v_{k'}'$, we have that $[(u_k,v_k,v_{k'},u_{k'})] = [(u_k',v_k',v_{k'}',u_{k'}')] = 0$, which by the existence of faces $u_1u_ku_{k'}$ and $v_1v_kv_{k'}$ implies that $[(u_1,v_1,v_k,u_k)] = [(u_1,v_1,v_{k'},u_{k'})]$, also a contradiction to the linear independence of $B$.  This proves the corollary.
\end{proof}

\section{Results on second homology}
\label{Hom2}
In this section, we prove upper bounds on $M_{2,2}(n)$ and $M_{2,d}(n)$ and a lower bound on $M_{2,5}(n)$.  For our first major result, we consider point configurations in $\mathbb{R}^2$.  If $p \in \mathbb{R}^2$, $x(p)$ denotes the $x$-coordinate of $p$.

\begin{theorem}
\label{M22}
There exists a constant $D$ so that $M_{2,2}(n) \leq Dn$.
\end{theorem}

We need two lemmas before we prove Theorem \ref{M22}.  Both the statement and the proof of our first lemma are found as \cite[Proposition 2.1]{Planar}.  The second lemma is a claim about arrangements of points that are close together.
\begin{lemma}
\label{CrossingLemma}
Let $S = \{u_1,u_2,v_1,v_2\} \subset \mathbb{R}^2$ so that $R(S)$ contains edges $u_1v_1$ and $u_2v_2$, and suppose that the line segments joining $u_1,v_1$ and $u_2,v_2$ intersect in $\mathbb{R}^2$.  Then $R(S)$ is a cone.
\end{lemma}
\begin{proof}
Let $p$ be the point of intersection between $u_1v_1$ and $u_2v_2$.  Suppose without loss of generality that the segment $pu_1$ is not longer than any of $pu_2, pv_1$, or $pv_2$.  Since $||pu_2||+||pv_2|| \leq 1$, then $||pu_1||+||pu_2|| \leq 1$ and $||pu_1||+||pv_2|| \leq 1$.  It follows from the triangle inequality that $u_1u_2$ and $u_1v_2$ are edges in $R(S)$ and hence $R(S)$ is a cone.
\end{proof}

\begin{lemma}
\label{PerpLemma}
Let $U$ and $V$ be finite sets of points in $\mathbb{R}^2$ such that all points of $U$ and $V$ are within distance $\epsilon$ of points $p_U$ and $p_V$ with $\ds(p_U,p_V) = 1$.  Choose $v_1 \neq v_2 \in V$.  Consider the vectors $w_1 := p_V-p_U$ and $w_2 := \frac{v_2-v_1}{\ds(v_1,v_2)}$, with $w_1 \cdot w_2$ denoting the standard scalar product.  Then one of the following is true. \newline
1) Either $\ds(v_1,u) \leq \ds(v_2,u)$ for all $u \in U$ or $\ds(v_1,u) \geq \ds(v_2,u)$ for all $u \in U$. \newline
2) There exists $\alpha = \alpha(\epsilon) \rightarrow 0$ as $\epsilon \rightarrow 0$ such that $|w_1 \cdot w_2| < \alpha$.
\end{lemma}
Roughly speaking, the second condition asserts that $w_1$ and $w_2$ are almost perpendicular.

\begin{proof}
By applying an isometry, we may assume without loss of generality that $p_U = (0,1)$ and $p_V = (0,0)$.  By applying a translation and replacing $\epsilon$ with $2 \epsilon$, we may also assume that $v_1 = (0,0)$.  Let $v_2 = (x,y)$.  Suppose that the first statement is false; that is, there exist $(x',1+y'), (x'',1+y'') \in U$ such that $\ds(v_1,(x',1+y')) > \ds(v_2,(x',1+y'))$ and $\ds(v_1,(x'',1+y'')) < \ds(v_2,(x'',1+y''))$.  Note that $|x|,|y|,|x'|,|y'|,|x''|,|y''| \leq \epsilon$.  We show that the second condition holds.

By considering squares of distances and simplifying, we have that $0 > x^2-2xx'-2y+y^2-2yy'$ and $0 < x^2-2xx''-2y+y^2-2yy''$.  This is impossible if $x=0$, which we see by dividing each side by $y$ and considering the fact that $y,y',y''$ are all close to $0$.  Then let $y=mx$.  Then we have that the quantities $x-2x'-2m+m^2x-2my' = x-2x'+m(-2+y-2y')$ and $x-2x''+m(-2+y-2y'')$ have opposite signs, which implies that $|m| < \alpha$ for some $\alpha \rightarrow 0$ as $\epsilon \rightarrow 0$.  Then $w_1$ is a vertical vector, $w_2$ is a nearly horizontal vector, and the result follows.
\end{proof}

\proofof{Theorem \ref{M22}}
Let $S$ be a point configuration in $\mathbb{R}^2$ with $|S| \leq n$.  Consider $0 < \epsilon < 2^{-1/2}$, and let $K$ be an $\epsilon$-cube such that $|K \cap S|$ is maximal.  Set $W := K \cap S$.  Since $\epsilon < 2^{-1/2}$, if $v \in V(\lk(w))$ for some $w \in W$, then $v$ is of distance no more than $3/2$ from the center of $K$.  There exists a value $\kappa$, which depends only on $\epsilon$, and $\epsilon$-cubes $\mathcal{K} = \{K=K_1, \ldots, K_\kappa\}$ such that for every $w \in W$, $\lk(w)$ contains only vertices in $S \cap (\cup K_i)$.  For each $w \in W$, let $$E_w = E(\lk(w)) = \cup_{1 \leq i < j \leq \kappa} E_{i,j,w}$$ be a set of edges as guaranteed by Corollary \ref{EdgeSet} with corresponding graphs $G^{i,j}_w$.  We take $r = 3/2$ in the corollary.

We claim that there exists an absolute constant $D'$ such that, for all $1 \leq i < j \leq \kappa$, $\sum_{w \in W} |E_{i,j,w}| \leq D'|W|$.  Assuming this claim, it then follows that $$\sum_{w \in W}|E_w| \leq {\kappa \choose 2}D'|W|,$$ and that there exists some $w \in W$ such that $|E_w| \leq {\kappa \choose 2}D'$.  By construction of $E_w$, there exists a constant $D$ such that $\B_1(\lk(w)) \leq D$.  The theorem follows by Lemma \ref{LinkLemma} and induction on $|S|$.  We prove the claim in two cases: the $i=1$ case and the $i>1$ case.

\noindent
\textbf{Case 1: $\mathbf{i=1}$:}

First suppose that $i=1$.  Let $U := S \cap K_j$.  By choosing $\epsilon$ sufficiently small and translating the coordinate system, we may assume that all points of $W$ are within distance $0.01$ of $(0,0)$.  If $\ds(u,w) > 1$ for all $w \in W, u \in U$, then $|E_{1,j,w}|=0$ for all $w$.  If $\ds(u,w) \leq 1$ for all $w \in W, u \in U$, then $|E_{1,j,w}| \leq 1$ for all $w$ by Condition 3 of Corollary \ref{EdgeSet} and the observation that $G^{1,j}_w$ is a complete bipartite graph.  Hence $\ds(u,w) > 1$ for some $u \in U, w \in W$ and $\ds(u',w') \leq 1$ for some $u' \in U, w' \in W$.  By rotating the coordinate system about the origin, we may assume that all points of $U$ are within distance $0.1$ of $(0,1)$.

Let $U_w$ be the set of endpoints of edges in $E_{1,j,w}$ that are in $U$.  If $w, w' \in W$ and $\ds(u,w') \leq \ds(u,w)$ for all $u \in U$, then there is an edge joining $w'$ to all $u \in U_w$ in $G^{i,j}_w$, which implies that $G^{i,j}_w$ is connected, and by Condition 3 of Corollary \ref{EdgeSet}, $|U_{w}| \leq 1$.  Construct $\tilde{W}$, starting from $W$, in the following way: whenever there is a pair $w \neq w' \in W$ such that $\ds(u,w') \leq \ds(u,w)$ for all $u \in U$, delete $w$, and continue until no more points can be deleted in this manner.  If $\epsilon$ is sufficiently small, then for all $w,w' \in \tilde{W}$, the slope $m$ of the line joining $w$ and $w'$ satisfies $-1 < \alpha < 1$; otherwise either $w$ or $w'$ would have been deleted by Lemma \ref{PerpLemma}.  It suffices to show that $\sum_{w \in \tilde{W}}|U_w| \leq D'|W|$ for some constant $D'$ by $$\sum_{w \in W}|E_{1,j,w}| = \sum_{w \in W}|U_w| \leq \sum_{w \in \tilde{W}}|U_w|+|W|.$$

Choose $u, u' \in U$.  If $\ds(u,w) \leq \ds(u',w)$ for all $w \in W$, then whenever $w'u' \in E_{1,j,w}$ for some $w$, $w'u$ is an edge in $G^{i,j}_w$ in the same component as $w'u'$.  Hence we may replace $w'u'$ with $w'u$ and still satisfy the conditions of Corollary \ref{EdgeSet}.  Construct $\tilde{U}$, starting from $U$, by deleting $u'$ for every pair of vertices $u \neq u' \in \tilde{U}$ such that $\ds(u,w) \leq \ds(u',w)$ for all $w \in W$, until no more vertices can be deleted in this manner.  We may choose $E_{1,j,w}$ so that every endpoint of an edge in $E_{1,j,w}$ in $U$ is actually in $\tilde{U}$.  Label the vertices of $\tilde{W}$ as $\{w_1, \ldots, w_{|\tilde{W}|}\}$ in order of ascending $x$-coordinates, and likewise label the vertices of $\tilde{U}$ as $\{u_1,\ldots,u_{|\tilde{U}|}\}$ in order of ascending $x$-coordinates.  As above, we may choose $\epsilon$ so that for all $u \neq u' \in \tilde{U}$, the slope $m$ of the line that joins $u$ and $u'$ satisfies $-1 < m < 1$.

Choose $i_1 < i_2$ and suppose that there exist $j_1 < j_2 < j_3 < j_4 < j_5 < j_6$ such that $u_{j_1}, u_{j_2},u_{j_3} \in U_{w_{i_2}}$ and $u_{j_4}, u_{j_5}, u_{j_6} \in U_{w_{i_1}}$.  Suppose that there exist $u_{j_1}w_a, u_{j_2}w_b, u_{j_3}w_c \in E_{i,j,w_{i_1}}$, and we derive a contradiction.  At most one of $w_a,w_b,w_c$ is equal to $w_{i_1}$.  Then there exists $k \in \{j_1,j_2,j_3\}$ such that $w_{i_1}u_k$ is not an edge; otherwise, $\lk(w_{i_1})$ contains two edges of $E_{i,j,w_{i_2}}$, a contradiction to Condition 4 of Corollary \ref{EdgeSet}.  Likewise, there exists $k' \in \{j_4,j_5,j_6\}$ such that $w_{i_2}u_{k'}$ is not an edge.  In particular, this shows that $|U_{w_{i_1}} \cap U_{w_{i_2}}| \leq 2$ for all $i_1 < i_2$.  The points $w_{i_1}$ and $u_{k'}$ are on opposite sides of the line joining $w_{i_2}$ and $u_k$ by consideration of the slopes of the lines joining the points, and similarly $w_{i_2}$ and $u_k$ are on opposite sides of the line joining $w_{i_1}$ and $u_{k'}$.  Hence the segments $w_{i_1}u_{k'}$ and $w_{i_2}u_k$ intersect in $\mathbb{R}^2$, and the set $\{w_{i_1},w_{i_2},u_k,u_{k'}\}$ violates Lemma \ref{CrossingLemma}.  Thus there cannot exist such $j_1 < \ldots < j_6$.

Let $W' = \{w \in \tilde{W}: |U_w| > 5\}$.  It suffices to show that $\sum_{w \in W'}|U_w| \leq D'|W|$ for some constant $D'$ by $$\sum_{w \in \tilde{W}}|U_w| \leq \sum_{w \in W'}|U_w|+5|W|.$$ For $w \in W'$, let $r(w)$ and $r'(w)$ be the indices of the points of $U_w$ with third smallest and second largest $x$-coordinates respectively.  By the above, if $w, w' \in W$ and $x(w') > x(w)$, then $$r(w') \geq r'(w) \geq r(w)+|U_w|-5.$$ If $w^-$ and $w^+$ are the points in $W'$ with smallest and largest $x$-coordinates respectively, then $$r(w^-)+\sum_{w^+ \neq w \in W'}(|U_w|-5) \leq r(w^+) \leq |\tilde{U}|-|U_{w^+}|+3.$$ Then $$\sum_{w \in W'}|U(w)| \leq |\tilde{U}|+5(|W'|-1)+3 \leq 5|W|.$$ This proves the result in the case that $i=1$.

\noindent
\textbf{Case 2: $\mathbf{i>1}$:}

Now fix $i$ and $j$ with $j > i > 1$.  Set $U := S \cap K_i$ and $V := S \cap K_j$, and for all $w \in W$, define $U_w := \{u_{w,1}, \ldots, u_{w,r_w}\}$ and $V_w := \{v_{w,1}, \ldots, v_{w,r_w}\}$ so that $E_{i,j,w}=\{u_{w,1}v_{w,1},\ldots,u_{w,r_w}v_{w,r_w}\}$.

If $\ds(u,v) > 1$ for all $u \in U, v \in V$, or if $\ds(w,u) > 1$ for all $w \in W, u \in U$, or if $\ds(w,v) > 1$ for all $w \in W, v \in V$, then $|E_w| = 0$ for all $w \in W$ and the result is proven.  If $\ds(u,v) \leq 1$ for all $u \in U, v \in V$, then $G^{i,j}_w$ is a complete bipartite graph and hence $|E_{i,j,w}| \leq 1$ for all $w \in W$ by Condition 3 of Corollary \ref{EdgeSet} and the claim is proven.  If $\ds(w,u) \leq 1$ for all $w \in W, u \in U$, then consider $v \in V_w \cap V_{w'}$ for $w \neq w'$ so that $uv \in E_{i,j,w}$.  Then $u$ and $v$ are both vertices in $\lk(w')$, which contradicts Condition 2 of Corollary \ref{EdgeSet} for $E_{i,j,w}$.  Hence $V_w \cap V_{w'} = \emptyset$, which implies that $\sum_{w \in W}|E_{i,j,w}| \leq |V| \leq |W|$, proving the claim.  Likewise, if $\ds(w,v) \leq 1$ for all $w \in W, v \in V$, then the claim is proven.  All pairs of points in $U \cup V \cup W$ that are not both in the same $\epsilon$-cube have distance between $1-4\epsilon$ and $1+4\epsilon$.  By choosing $\epsilon$ sufficiently small and making a suitable isometric change of coordinates, we may assume that all vertices of $U,V,W$ are within distance $0.01$ of $(0,0)$, $(0,1)$, and $(\sqrt{3}/2, 1/2)$ respectively.

For $u \in U$ and $v \in V$, let $W_u = \{w \in W: u \in U_w\}$ and $W_v = \{w \in W: v \in V_w\}$.  For $w \in W_u$, define the vertex $v(u,w)$ so that the edge $\{u,v(u,w)\} \in E_{i,j,w}$.  If $w, w' \in W_u$, then either the line that joins $w$ and $w'$ has slope $m$ satisfying $\sqrt{3}-0.1 < m < \sqrt{3}+0.1$; or either $\ds(w,v) \leq \ds(w',v)$ for all $v \in V$, or $\ds(w',v) \leq \ds(w,v)$ for all $v \in V$ by Lemma \ref{PerpLemma}.  Without loss of generality, assume the former.  Then $\{u,v(u,w')\}$ is an edge in $\lk(w)$ and in $\lk(w')$, a contradiction to Condition 2 of Corollary \ref{EdgeSet}.  The vertices in $W_u$ can then be arranged $w_{u,1}, \ldots, w_{u,s_u}$ in order of increasing distance from $U$.  By the same argument, the vertices of $W_v$ can be similarly arranged $w_{v,1}, \ldots, w_{v,t_v}$ in order of increasing distance from $V$.

For all $u \in U$ with $W_u \neq \emptyset$, there exists a vertex $v(u) \in V$ and $w \in W_u$ such that $\{u,v(u)\} \in E_{i,j,w}$ and $\ds(w,u) \geq \ds(w',u)$ for all $w' \in W_u$.  Likewise, for all $v \in V$ with $W_v \neq \emptyset$, there exists a vertex $u(v) \in U$ such that $\{u(v),v\} \in E_{i,j,w}$ and $\ds(w,v) \geq \ds(w',v)$ for all $w' \in W_v$.  There are at most $|U|$ (or $|V|$) edges $uv$ in $\cup_{w \in W}E_{i,j,w}$ such that $v=v(u)$ (or $u=u(v)$).  Also, the $E_{i,j,w}$ are disjoint by Condition 2 of Corollary \ref{EdgeSet}.  Hence if $\sum_{w \in W}|E_{i,j,w}| > 2|W| \geq |U|+|V|$, there exist $w \in W, u \in U, v \in V$ such that $uv \in E_{i,j,w}$, $u \neq u(v)$, and $v \neq v(u)$.  In this case, choose $w_u \in W_u, w_v \in W_v$ such that $\ds(u,w_u) > \ds(u,w)$ and $\ds(v,w_v) > \ds(v,w)$.  By consideration of the slopes between the points $u,v,w,w_u,w_v$, the points $v$ and $w_v$ are on opposite sides of the line joining $u$ and $w_u$, and $u$ and $w_u$ are on opposite sides of the line joining $v$ and $w_v$, and so the segments $uw_u$ and $vw_v$ intersect.  By Lemma \ref{CrossingLemma}, either $uw_v$ or $vw_u$ is an edge, yielding either the face $uvw_v$ or $uvw_u$.  This contradicts Condition 2 of Lemma \ref{EdgeSet}.  We conclude that $\sum_{w \in W}|E_{i,j,w}| \leq 2|W|$ as desired.
\endproof

\begin{theorem}
\label{M2D}
For all fixed $\delta > 0$ and $n$ sufficiently large, $M_{2,d}(n) < \delta n^2$.
\end{theorem}

Before we give the proof, we need two additional lemmas.  The first concerns bipartite graphs that avoid certain kinds of subgraphs.

\begin{lemma}
\label{MooreBound}
Let $G$ be a bipartite graph on vertices $U \sqcup V$ with $|U| \leq n$ and $|V| \leq n$.  Suppose that no two vertices of $U$ share three common neighbors.  Then there exists a constant $C$ such that $G$ has at most $Cn^{3/2}$ edges.
\end{lemma}
\begin{proof}
Equivalent to the condition that no two vertices of $U$ share three common neighbors is the condition that no three vertices of $V$ share two common neighbors.  For each $v \in V$, let $N(v)$ be the set of neighbors of $v$.  Let ${N(v) \choose 2}$ be the set of pairs of neighbors of $v$, so that $|{N(v) \choose 2}| = {|N(v)| \choose 2}$.  Also, let $d$ be the average degree of vertices in $v$.  Then $\sum_{v \in V} |{N(v) \choose 2}| \geq |V|{d \choose 2}.$  Since no three vertices in $V$ share two common neighbors, it must be that $|V|{d \choose 2} \leq 2{|U| \choose 2} \leq 2{n \choose 2}$ by the pigeonhole principle.  There exists a constant $C$ such that $d \leq C n|V|^{-1/2}$, and hence $G$ has at most $C n|V|^{1/2}$ edges.  This proves the result by $|V| \leq n$.
\end{proof}

The second lemma concerns induced matchings.  Let $G$ be a bipartite graph with vertex sets $U$ and $V$.  Then a \textit{matching} $M$ is a set of edges in $G$ such that no two edges have a common endpoint.  We say that $M$ is an \textit{induced matching} if whenever $uv, u'v' \in M$ for $u,u' \in U, v,v' \in V$, $G$ does not contain edges $uv'$ or $u'v$.  The following is an immediate consequence of \cite[Proposition 10.45]{TaoVu}.

\begin{lemma}
\label{InducedMatchings}
Let $G$ be a bipartite graph with vertex sets $U$ and $V$, $|U| \leq n$ and $|V| \leq n$.  Let $M_1, \ldots, M_t, t \leq n$ be disjoint sets of edges that are each an induced matching in $G$.  Let $\delta > 0$ be fixed.  Then $\sum_{i=1}^t|M_i| < \delta n^2$ if $n$ is sufficiently large.
\end{lemma}

\proofof{Theorem \ref{M2D}}
We use some of the same methods as in the proof of Theorem \ref{M22}.  Let $S$ be a point configuration in $\mathbb{R}^d$ with $|S| \leq n$.  Fix $\epsilon = d^{-1/2}$, and let $K$ be an $\epsilon$-cube such that $W := |K \cap S|$ is maximal.  There is a value $\kappa$, which depends only on $d$, and set of $\epsilon$-cubes $\mathcal{K} = \{K=K_1, \ldots, K_\kappa\}$ such that every vertex in the link of each $w \in W$ is contained in $S \cap (\cup K_i)$.  For each $w \in W$, let $$E_w = E(\lk(w)) = \cup_{1 \leq i < j \leq \kappa} E_{i,j,w}$$ be a set of edges as guaranteed by Corollary \ref{EdgeSet} with $r=3/2$.

We show that for any given $\delta' > 0$ and $n$ sufficiently large, for all $1 \leq i < j \leq \kappa$, $\sum_{w \in W} |E_{i,j,w}| < \delta'|W|n$.  It then follows that $\sum_{w \in W}|E_w| < {\kappa \choose 2}\delta'|W|n$, and that there exists some $w \in W$ such that $|E_w| < {\kappa \choose 2}\delta'n$.  By construction of $E_w$, $\B_1(\lk(w)) < \delta n$ for a $\delta > 0$ that can be chosen arbitrarily small by choosing $\delta'$ sufficiently small.  Then by Lemma \ref{LinkLemma}, $\B_2(R(S)) < \B_2(R(S-\{w\}))+ \delta n$.  By induction on $|S|$ (we keep $n$ fixed and decrease $|S|$ in the inductive step), $\B_2(R(S)) < \delta n^2$ as desired.

First consider the case that $i=1$.  Let $U := S \cap K_j$.  For each $w \in W$, let $U_w$ be the set of endpoints of edges in $E_{1,j,w}$ that are in $U$, as in the proof of Theorem \ref{M22}.  By the same argument as in the proof of Theorem \ref{M22}, $|U_w \cap U_{w'}| \leq 2$ for all $w \neq w'$.  Construct a bipartite graph $G$ with vertex set $W \sqcup U$ and an edge $wu, w \in W, u \in U$ whenever $u \in U_w$.  Label the edge set of $G$ by $EG$.  By $|U_w \cap U_{w'}| \leq 2$ for all $w \neq w'$, no two vertices in $W$ has three common neighbors in $G$.  It follows from Lemma \ref{MooreBound} and the fact that $|U| \leq |W|$ that $$\sum_{w \in W} |E_{1,j,w}| = |EG| \leq C|W|^{3/2} < \delta'|W|n$$ for some constant $C$.  The last inequality follows by taking $n$ sufficiently large.

Now suppose that $i > 1$.  Set $U := S \cap K_i$ and $V := S \cap K_j$, and for all $w \in W$, define $U_w = \{u_{w,1}, \ldots, u_{w,r_w}\}$ and $V_w = \{v_{w,1}, \ldots, v_{w,r_w}\}$ so that $E_{i,j,w}=\{u_{w,1}v_{w,1},\ldots,u_{w,r_w}v_{w,r_w}\}$.  Let $G'$ be the bipartite graph on vertices $U \sqcup V$ with an edge $uv, u \in U, v \in V$ whenever $uv$ is an edge in $R(S)$.  Conditions 1 and 3 of Lemma \ref{EdgeSet} imply that $E_{i,j,w}$ is a matching in $G'$ for all $w \in W$.  If $E_{i,j,w}$ contains edges $uv$ and $u'v'$, and there is an edge $uv'$ or $u'v$ in $G'$, then $uv$ and $u'v'$ are in the same component in $G' \cap \lk(w)$.  Hence Condition 3 of Lemma \ref{EdgeSet} implies that $E_{i,j,w}$ is in fact an induced matching.  It must be that $E_{i,j,w} \cap E_{i,j,w'} = \emptyset$ for all $w \neq w'$; otherwise $\lk(w)$ contains an edge of $E_{i,j,w'}$, which violates Condition 2 of Corollary \ref{EdgeSet}.  If $|W| < \delta'n$, then $$\sum_{w \in W} |E_{i,j,w}| \leq |W|^2 < \delta'|W|n.$$ Otherwise, it follows from Lemma \ref{InducedMatchings} that $$\sum_{w \in W} |E_{i,j,w}| < \delta'|W|^2 < \delta'|W|n$$ for sufficiently large $n$.
\endproof

\begin{theorem}
\label{n32}
There exists a constant $C$ such that for sufficiently large $n$, $M_{2,5}(n) > C n^{3/2}$.
\end{theorem}
\begin{proof}
We establish the result by producing a point configuration $S \subset \mathbb{R}^5$ with at most $n$ vertices and with $\B_2(R(S)) > Cn^{3/2}$.  Let $k$ be the largest integer such that $3k^2 \leq n$.  Choose a value $\delta$ small relative to $n$ and a value $\epsilon$ small relative to $\delta$; we may take $\delta = n^{-1}$ and $\epsilon = n^{-3}$.  Let $U = \{u_{i,j}, 1 \leq i,j \leq k\}$, $V = \{v_{i,j}, 1 \leq i,j \leq k\}$, $W = \{w_{i,j}, 1 \leq i,j \leq k\}$ with $$u_{i,j} = \left(\frac{\sqrt{2}}{2}\cos (i\delta),\frac{\sqrt{2}}{2}\sin (i\delta),0,0,j \epsilon \right),$$ $$v_{i,j} = \left(0,0,\frac{\sqrt{2}}{2}\cos (i\delta),\frac{\sqrt{2}}{2}\sin (i\delta),j \epsilon \right),$$ $$w_{i,j} = \left(\frac{\sqrt{2}}{4}\cos (i\delta),\frac{\sqrt{2}}{4}\sin (i\delta),\frac{\sqrt{2}}{4}\cos (j\delta),\frac{\sqrt{2}}{4}\sin (j\delta),\frac{\sqrt{3}}{2} \right).$$ 
\noindent
The edge set of $R(S)$ is exactly the following:
\begin{enumerate}
  \item $u_{i,j}u_{i',j'}$ for all $1 \leq i,j,i',j' \leq k$,
  \item $v_{i,j}v_{i',j'}$ for all $1 \leq i,j,i',j' \leq k$,
  \item $w_{i,j}w_{i',j'}$ for all $1 \leq i,j,i',j' \leq k$,
  \item $u_{i,j}v_{i',j}$ for all $1 \leq i,i',j \leq k$,
  \item $u_{i,j}w_{i,j'}$ for all $1 \leq i,j,j' \leq k$,
  \item $v_{i,j}w_{i',i}$ for all $1 \leq i,i',j \leq k$.
\end{enumerate}
The non-existence of edges $u_{i,j}w_{i',j'}$ for $i' \neq i$ and $v_{i,j}w_{i',i''}$ for $i'' \neq i$ is guaranteed by a sufficiently small choice of $\epsilon$.  For all $w \in W$, the set edges of $\lk(w)$ with one endpoint in $U$ and the other in $V$ constitutes an induced matching.  Furthermore, these matchings are disjoint over all $w$.

It can be verified that $\beta_2(R(S)) \geq b$, with $b \approx 3^{-3/2}n^{3/2}$.  We defer the details of this calculation to the more general setting of Section \ref{QuasiRips}.
\end{proof}

Label the above construction with $\delta = 1/n$ and $\epsilon = n^{-3}$ as $S^2(n)$.

\section{Results on higher homology}
The results of the previous section can be extended to higher Betti numbers.  In this section we prove two such extensions.
\label{Hom3}
\begin{theorem}
Let $p \geq 2$, $d$, and $\delta > 0$ be fixed.  If $n$ is sufficiently large, then $M_{p,d}(n) < \delta n^p$.  Also, there exists a value $D_p$ which depends only on $p$ such that $M_{p,2}(n) \leq D_p n^{p-1}$.
\end{theorem}
\begin{proof}
We prove the first statement by induction on $p$.  The case that $p=2$ follows from Theorem \ref{M2D}.  Let $S \subset \mathbb{R}^d$ with $|S| \leq n$.  For $p > 2$, assume that $n$ is large enough so that $M_{p-1,d}(n) < \delta n^{p-1}$.  Choose $v \in S$.  Then by the inductive hypothesis, $\B_{p-1}(\lk(v)) < \delta n^{p-1}$.  We calculate that $\B_p(R(S)) < \delta|S|n^{p-1} \leq \delta n^p$ by induction on $|S|$.  Indeed, it follows from Lemma \ref{LinkLemma} that $\B_p(R(S)) \leq \B_p(R(S-v)) + \B_{p-1}(\lk(v)) < \delta|S-v|n^{p-1} + \delta n^{p-1}$ as desired.

The second statement follows from Theorem \ref{M22} in the same way.
\end{proof}

Let $\Gamma$ and $\Delta$ be two simplicial complexes.  We define their \textit{simplicial join} $\Gamma \ast \Delta$ by $V(\Gamma \ast \Delta) := V(\Gamma) \sqcup V(\Delta)$ and faces $\{F \cup G: F \in \Gamma, G \in \Delta\}$.  Let $S,S' \subset \mathbb{R}^d$ such that for all $s \in S, s' \in S'$, $\ds(s,s') \leq 1$.  Then $R(S \cup S') = R(S) \ast R(S')$.  By the K\"{u}nneth Formula, $\B_p(R(S \cup S')) = \sum_{i+j=p-1}\B_i(R(S))\B_j(R(S'))$.

\begin{lemma}
For every $k > 0$ and $d \geq 2$, there exists a constant $C_k$ such that $M_{2k-1,d}(n) \geq C_k n^k$ for sufficiently large $n$.
\end{lemma}
\begin{proof}
We prove the result by giving a point configuration $S \subset \mathbb{R}^2$ with $|S| \leq n$ and $\B_{2k-1}(R(S)) \geq C_kn^k$.  Let $r := \lfloor n/(2k) \rfloor$, $\theta := 1/n$, and $\epsilon := n^{-4}$.  For $1 \leq i \leq k$, define $s_i^+ := (1/2,i \epsilon)$ and $s_i^- := (-1/2,i \epsilon)$.  Let $S_1 := \{s_1^+, \ldots, s_r^+, s_1^-, \ldots, s_r^-\}$.  For $2 \leq j \leq p$, construct $S_j$ by rotating $S_1$ counterclockwise about the origin by an angle of $\theta j$ and let $S := S_1 \cup \ldots \cup S_k$.

For each $i$, $\B_1(R(S_i)) = r-1$ by Lemma \ref{BipartiteLemma}.  For all $i \neq j$ and $s \in S_i, s' \in S_j$, $\ds(s,s') < 1$ by the small choice of $\epsilon$.  It follows by the K\"{u}nneth Formula that $\B_{2k-1}(R(S)) \geq (r-1)^k$.
\end{proof}

Label the above construction as $S^{2k-1}(n)$ with $S^{-1}(n) = \emptyset$.

\begin{theorem}
For every $p > 0$ and $d \geq 5$, there exists a constant $C_p$ such that $M_{p,d}(n) \geq C_p n^{p/2+1/2}$ for sufficiently large $n$.
\end{theorem}
\begin{proof}
The result follows for odd $p$ by the existence of $S^{2k-1}(n)$, so consider even $p$.  Let $S = S^2(\lfloor n/2 \rfloor)$ and $\tilde{S} = S^{p-3}(\lceil n/2 \rceil)$.  Let $S'$ be the image of $\tilde{S}$ under the isometry that sends $(x,y)$ to $(\sqrt{2}/4,x,\sqrt{2}/4,y,\sqrt{3}/6)$.  There exists $\alpha \rightarrow 0$ as $n \rightarrow \infty$ such that every point in $S$ is within distance $\alpha$ from either $(0,0,\sqrt{2}/2,0,0)$, $(\sqrt{2}/2,0,0,0,0)$, or $(\sqrt{2}/4,0,\sqrt{2}/4,0,\sqrt{3}/2)$, and every point of $S'$ is within distance $\alpha$ of either $(\sqrt{2}/4,1/2,\sqrt{2}/4,0,\sqrt{3}/6)$ or $(\sqrt{2}/4,-1/2,\sqrt{2}/4,0,\sqrt{3}/6)$.  Hence for all $s \in S, s' \in S'$, $\ds(s,s') < \sqrt{7/12}+\alpha'$, where $\alpha' \rightarrow 0$ as $n \rightarrow \infty$.  This proves the result by the K\"{u}nneth Formula.
\end{proof}

\section{Quasi-Rips complexes}
\label{QuasiRips}
Quasi-Rips complexes, discussed in \cite{Planar}, are relaxations of Rips complexes.  Given a finite set $S \subset \mathbb{R}^d$ and fixed $0 < \alpha < 1$, a \textit{quasi-Rips complex with parameter} $\alpha$ on $S$ is a flag complex with vertex set $S$, an edge $uv$ whenever $\ds(u,v) \leq \alpha$, and no edge $uv$ when $\ds(u,v) > 1$.  If $\alpha < \ds(u,v) \leq 1$, then the edge $uv$ may be included or excluded arbitrarily.  All Rips complexes are quasi-Rips complexes with parameter $\alpha$ for any $0 < \alpha < 1$.

There is much greater freedom in the kinds of graphs that arise as the graphs of quasi-Rips complexes.  Let $G$ be a graph with three vertex subsets $U_1,U_2,U_3$ such that for $1 \leq i \leq 3$, there is an edge $uu'$ for all $u,u' \in U_i$.  The other edges of $G$ may be chosen arbitrarily.  For any $0 < \alpha < 1$, $G$ can arise as a quasi-Rips complex of a point configuration in $\mathbb{R}^2$ if the points of $U_1, U_2, U_3$ are all near $(0,0), (0,1), (\sqrt{3}/2,1/2)$ respectively and inside the triangle with these three vertices.  If $\alpha < 1/2$, $G$ can even be the graph of a quasi-Rips complex of a point configuration in $\mathbb{R}^1$ by concentrating all points of $U_1, U_2, U_3$ near $0, 1/2, 1$ respectively and inside the interval $[0,1]$.

Despite this freedom, the Betti numbers of quasi-Rips complexes obey nontrivial upper bounds.  Given $S \subset \mathbb{R}^d$, let $Q^{\alpha}(S)$ be the set of quasi-Rips complexes on $S$ with parameter $\alpha$.  Let $Q^{\alpha}_d(n)$ be the union of all $Q^{\alpha}(S)$ as $S$ ranges over subsets of $\mathbb{R}^d$ of size at most $n$, and define $$M^{\alpha}_{p,d}(n) := \max\{\B_p(\Gamma): \Gamma \in Q^{\alpha}_d(n)\}.$$

We focus specifically on $M^{\alpha}_{2,d}(n)$.  To do so, we again consider induced matchings.  Let $I(n)$ be the maximum value of $\sum_{i=1}^t |M_i|$, where $t \leq n$ and the $M_i$'s are disjoint, induced matchings on a bipartite graph $G$ with $V(G) = U \sqcup V$ and $|U|, |V| \leq n$.

\begin{theorem}
\label{qr}
For each $d \geq 2$ and $0 < \alpha < 1$, there exist values $C_{d,\alpha}$ and $D_{d,\alpha}$ such that $$C_{d,\alpha}I(n) < M^{\alpha}_{2,d}(n) < D_{d,\alpha}I(n)$$ for sufficiently large $n$.
\end{theorem}
\begin{proof}
The proof of the upper bound is very similar to that of Theorem \ref{M2D}.  The changes necessary are to use $\epsilon = \alpha d^{-1/2}$ instead of $\epsilon = d^{-1/2}$, and to make the observation that for some absolute constant $C$ and all $n>n' \gg 0$, $I(n) \geq CI(n') n/n'$.

For the lower bound, consider a bipartite graph $G$ with vertex set $U \sqcup V$, $|U|,|V| \leq n$, and disjoint, induced matchings $M_1, \ldots, M_t, t \leq n$ such that $\sum_{i=1}^t |M_i| = I(n)$.  Let $\mathcal{N} := \{N_1, N_2, \ldots, N_{t'}\}, t' = \min\{t,\lfloor n/3 \rfloor\}$ be a subset of $t'$ largest matchings of the set $\{M_1, \ldots, M_t\}$.  Let $U'$ be a set of $\min\{|U|,\lfloor n/3 \rfloor\}$ vertices of $U$ that are endpoints for the largest number of matchings in $\mathcal{N}$, and restrict each element of $\mathcal{N}$ to edges with endpoints in $U' \sqcup V$.  Finally, let $V'$ be a set of $\min\{|V|,\lfloor n/3 \rfloor\}$ vertices of $V$ that are endpoints for the largest number of matchings in $\mathcal{N}$, and restrict each element of $\mathcal{N}$ to edges with endpoints in $U' \sqcup V'$.  Then for large $n$, $\sum_{i=1}^{t'} |N_i| \geq \frac{1}{27.1}I(n)$.

Let $G'$ be a graph with vertex set $U' \sqcup V' \sqcup \mathcal{N}$, with edges defined as follows.  For $u \in U', v \in V'$, $uv$ is an edge in $G'$ if it is an edge in $G$.  For all $u,u' \in U', v,v' \in V'$, $uu'$ and $vv'$ are edges in $G'$.  The neighbors of $N_i$ are all $N_j$ for $j \neq i$, and all vertices $u \in U', v \in V'$ such that $uv \in N_i$.  Then, by the discussion preceding Theorem \ref{qr}, there exists a simplicial complex $\Gamma$ such that $\Gamma \in Q^{\alpha}_d(n)$ and $\Gamma = X(G')$.  Next we calculate $\B_2(\Gamma)$.

Consider $\Gamma'$, which is obtained from $\Gamma$ by removing all faces of the form $Nuv$ for $N \in \mathcal{N}, u \in U', v \in V'$.  By construction, every face removed in this manner is a maximal face in $\Gamma$.  Note that $\Gamma'$ is neither a Rips complex nor a flag complex.  By Lemma \ref{BipartiteLemma}, $\B_1(\Gamma'[U',V']) = \B_1(\Gamma[U',V']) \leq \lfloor n/3 \rfloor$.  To calculate $\B_1(\Gamma')$, consider $\Gamma'_i := \Gamma'[U',V',N_1, \ldots, N_i]$.  Then $\lk_{\Gamma'_i}(N_i)$ consists of at most three components, as $\lk_{\Gamma'_i}(N_i) \cap \Gamma[U'], \lk_{\Gamma'_i}(N_i) \cap \Gamma[V'], \lk_{\Gamma'_i}(N_i) \cap \Gamma[\mathcal{N}]$ are all connected, and so $\B_0(\lk_{\Gamma'_i}(N_i)) \leq 2$.  By induction on $i$, $\B_1(\Gamma') \leq \lfloor n/3 \rfloor + 2\lfloor n/3 \rfloor \leq n$.  By Euler-Poincar\'{e} formula, $\B_2(\Gamma) > \frac{1}{27.1}I(n) - n \geq C_{d,\alpha}I(n)$ for some constant $C_{d,\alpha}$, which proves the lower bound.
\end{proof}

The example of Theorem \ref{n32} satisfies the description of Theorem \ref{qr}, with each vertex in $W$ corresponding to a matching of $k$ pairs of vertices on $U$ and $V$.  Hence $\B_2(R(S)) \approx 3^{-3/2}n^{3/2}.$

Determining the value of $I(n)$, even to within a multiplicative constant, is a very challenging problem.  It is shown in \cite{Elkin} that there exists a constant $C$ so that, for large $n$, there exists $A \subset \mathbb{Z}/n\mathbb{Z}$ such that $A$ contains no arithmetic progressions of length $3$ and $|A| \geq C n 2^{-2\sqrt{2}\sqrt{\log_2(n)}} \log^{1/4}(n)$.  Such an $A$ can be adapted into a tripartite graph with $n$ vertices in each component, at least $C n^2 2^{-2\sqrt{2}\sqrt{\log_2(n)}} \log^{1/4}(n)$ triangles, and the property that no two triangles that share an edge.  This graph can then be further adapted into $n$ disjoint, induced matchings on a bipartite graph with $n$ vertices on each side.  The collective size of the matchings is $C n^2 2^{-2\sqrt{2}\sqrt{\log_2(n)}} \log^{1/4}(n)$.  An upper bound on $I(n)$, as given in the proof of \cite[Proposition 10.45]{TaoVu}, is $C'\frac{n^2}{(\log_*(n))^{1/5}}$ for some constant $C'$, where $\log_*(n)$ is the number of natural logarithms one needs to apply to $n$ to obtain a nonpositive value.  The $\log_*(n)$ term comes from the usage of the Szemer\'edi Regularity Lemma in the proof.

Theorem \ref{qr} can be extended to higher Betti numbers using similar techniques as in Section \ref{Hom3}.

\end{document}